\documentclass[reqno]{amsart}
\newcommand{\R}{\mathbb{R}}

\newcommand{\N}{\mathbb{N}}

\newcommand{\ind}{\mathop\mathrm{ind}\nolimits}

\theoremstyle{plain}
\newtheorem{theorem}{Theorem}[section]
\newtheorem{corollary}[theorem]{Corollary}
\newtheorem{lemma}[theorem]{Lemma}
\newtheorem{proposition}[theorem]{Proposition}
\theoremstyle{definition}
\newtheorem{definition}[theorem]{Definition}
\newtheorem{remark}[theorem]{Remark}

\numberwithin{equation}{section}

\newcommand{\dist}{\mathop\mathrm{dist}\nolimits}

\newcommand{\cl}{\colon}

\begin{document}
\title[A continuation result for forced oscillations]% 
{A continuation result for forced oscillations of constrained motion problems with infinite delay}
\author[P.\ Benevieri]{Pierluigi Benevieri}
\author[A.\ Calamai]{Alessandro Calamai}
\author[M.\ Furi]{Massimo Furi}
\author[M.P.\ Pera]{Maria Patrizia Pera}

\date{}%\today}

\begin{abstract}
We prove a global continuation result for $T$-periodic solutions of a $T$-periodic parametrized second order retarded functional differential equation on a boundaryless compact manifold with nonzero Euler-Poincar\'e characteristic. The approach is based on the fixed point index theory for locally compact maps on ANRs.
As an application, we prove the existence of forced oscillations of
retarded functional motion equations defined on topologically nontrivial compact constraints.
This existence result is obtained under the assumption that the frictional coefficient is nonzero, and we conjecture that it is still true in the frictionless case.
\end{abstract}


\maketitle

%%%%%
%%%%% section 1: introduction
%%%%%

\section{Introduction}
\label{Introduction}
\setcounter{equation}{0}

Let $X \subseteq \R\sp s$ be a smooth boundaryless manifold,
and let $F\cl \R \times C((-\infty, 0],X) \to \R\sp s$ be a continuous map
such that
\[
F(t, \varphi) \in T_{\varphi(0)} X \,, \quad \forall \,
(t, \varphi) \in \R \times C((-\infty, 0],X),
\]
where, given $q \in X$, $T_qX \subseteq \R\sp s$ denotes the tangent space of $X$ at $q$.
Any tangent vector field with this property will be called a \emph{functional field} on $X$.

Assume that $F$ is $T$-periodic in the first variable
and consider the following \emph{retarded functional motion equation} on $X$,
depending on a parameter $\lambda \geq 0$:
\begin{equation}
\label{intro-eq-motion}
x_\pi''(t) = 
\lambda \left( F (t,x_t) - \varepsilon x'(t) \right),
\end{equation}
where $x''_\pi(t)$ stands for the tangential part of the acceleration
$x''(t) \in \R\sp{s}$ at the point $x(t) \in X$, $x_t$ is the function $\theta \mapsto x(t+\theta)$, and $\varepsilon \geq 0$ is the frictional coefficient.
Given $\lambda \geq 0$, any $T$-periodic solution of
\eqref{intro-eq-motion} is called a \emph{forced oscillation}
corresponding to the value $\lambda$ of the parameter.
Notice that when $\lambda=0$ the above equation reduces to the so-called \emph{inertial equation} and one obtains the geodesics of $X$ as solutions.

A pair $(\lambda,x)$, with $\lambda \geq 0$ and $x\cl \R \to X$ a forced oscillation of \eqref{intro-eq-motion} corresponding to $\lambda$, is called a 
\emph{$T$-forced pair} of \eqref{intro-eq-motion}.
The set of $T$-forced pairs is regarded as a subset of $[0,+ \infty)\times C\sp{1}_T(X)$, where $C\sp{1}_T(X)$ is
the metric subspace of the Banach space $C\sp{1}_T(\R\sp{s})$ of the
$T$-periodic $C\sp{1}$ maps from $\R$ to $X$.

Given $q \in X$, we denote by $\bar q \in C\sp{1}_T(X)$ the constant
map $t \mapsto q$, and we call \emph{trivial $T$-forced pair} a pair of the form $(0,\bar q)$.
An element $q_0 \in X$ will be called a \emph{bifurcation point} of the
equation \eqref{intro-eq-motion} if every neighborhood of
$(0, \bar q_0)$ in $[0,+ \infty) \times C\sp{1}_T(X)$ contains a nontrivial
$T$-forced pair.
We point out that, in this case, any nontrivial $T$-forced pair 
$(\lambda,x)$ of \eqref{intro-eq-motion}, sufficiently close to 
$(0, \bar q_0)$, must have $\lambda >0$.
This is due to the fact that there are no nontrivial closed geodesics too
close to a given point in a Riemannian manifold.

In this paper we investigate the structure of the set of $T$-forced pairs of \eqref{intro-eq-motion}.
Our main result, Theorem \ref{teo-rami2ord} below, states that
if $X$ is compact with nonzero
Euler--Poincar\'e characteristic, and if $F$ has bounded image and verifies a suitable
Lipschitz-type assumption, then there exists an unbounded connected branch of
nontrivial $T$-forced pairs whose closure intersects the set of the trivial
$T$-forced pairs.

We stress that, when $\varepsilon>0$, the bifurcating branch given in Theorem \ref{teo-rami2ord} is unbounded with respect to the parameter $\lambda$.
Therefore, in this case,
the retarded functional motion equation
\[
x_\pi''(t) = F (t,x_t) - \varepsilon x'(t)
\]
admits at least one forced oscillation (see Corollary \ref{coroll-esist} below).
This consequence of Theorem \ref{teo-rami2ord}
generalizes results given in
\cite{BCFP2} and \cite{BCFP4} for equations with constant time lag
(see also \cite{FuPe90} for the undelayed case).
On the opposite, in the frictionless case, the existence of an unbounded bifurcating branch is not sufficient to guarantee the existence of a forced oscillation of the equation
\begin{equation}
\label{intro-eq}
x_\pi''(t) = F (t,x_t).
\end{equation}

As far as we know, the
problem of the existence of forced oscillations of \eqref{intro-eq} is
still open, even in the undelayed case. An affirmative answer, in the
undelayed situation, regarding the special constraint $X = S \sp 2$ (the
spherical pendulum) can be found in \cite{FuPe91} (see also \cite{FuPe93o} for
the extension to the case $X = S\sp{2n}$).
Let us point out that Theorem \ref{teo-rami2ord} below will be crucial for further investigation of the question whether or not the forced spherical pendulum admits forced oscillations in the retarded functional case.

To get our main result, we consider a first order retarded
functional differential equation (RFDE) on the tangent bundle $TX
\subseteq \R\sp{2s}$, which is equivalent to the second order equation
\eqref{intro-eq-motion}.
More precisely, in the first and preliminary part of the paper we study first order 
parametrized RFDEs of the type
\begin{equation}
\label{intro-eq-RFDE}
x'(t) = f(\lambda, t,x_t),
\end{equation}
where $f\cl [0,+\infty) \times \R \times C((-\infty,0],M) \to \R\sp{k}$ is a parametrized functional field, $T$-periodic in the second variable, on a smooth manifold $M \subseteq \R\sp k$, possibly with boundary.
For this equation we prove a result, Theorem \ref{teo-ipotesilemma} below, 
regarding the existence of a global branch of pairs $(\lambda,x)$, where $x$ is a $T$-periodic solution of \eqref{intro-eq-RFDE} corresponding to the parameter $\lambda$. We point out that in Theorem \ref{teo-ipotesilemma}, $M$ is a manifold with boundary.
This fact will be crucial in the application to second order equations, as it will be clear in the proof of Theorem \ref{teo-rami2ord}.

Among the wide bibliography on RFDEs in Euclidean spaces we refer to the works of Gaines and Mawhin
\cite{GM}, Nussbaum \cite{Nu2, Nu3} and Mallet-Paret, Nussbaum
and Paraskevopoulos \cite{MNP}.
For RFDEs on manifolds we cite the papers of Oliva \cite{Ol1, Ol2}.
For general reference we suggest the monograph by Hale and Verduyn Lunel
\cite{HL}.

%%%%%
%%%%% section 2 - preliminary results + fixed point index
%%%%%

\section{Preliminaries}
\label{sect-preliminaries-index}
\setcounter{equation}{0}

\subsection{RFDE}

Let $M$ be an arbitrary subset of $\R\sp k$.
We recall the notions of tangent cone and tangent space of $M$ at a
given point $p$ in the closure $\overline M$ of $M$. The definition of tangent
cone is equivalent to the classical one introduced by Bouligand in \cite{Bou}.

\begin{definition}
\label{tangent cone}
A vector $v \in \R\sp k$ is said to be \emph{inward} to $M$ at
$p \in \overline M$ if there exist two sequences $\{\alpha_n\}$ in
$[0,+ \infty)$ and $\{p_n\}$ in $M$ such that
\[
p_n \to p \quad \mbox{and} \quad \alpha_n (p_n -p) \to v.
\]
The set $C_pM$ of the inward vectors to $M$ at $p$ is called the
\emph{tangent cone} of $M$ at $p$. The \emph{tangent space} $T_pM$ of $M$
at $p$ is the vector subspace of $\R\sp k$ spanned by $C_pM$. A vector $v$
of $\R\sp k$ is said to be \emph{tangent} to $M$ at $p$ if $v \in T_pM$.
\end{definition}

To simplify some statements and definitions we put $C_pM = T_pM = \emptyset$
whenever $p \in \R\sp k$ does not belong to $\overline M$ (this can be regarded as
a consequence of Definition \ref{tangent cone} if one replaces the assumption
$p \in \overline M$ with $p \in \R\sp k$). Observe that $T_pM$ is the trivial
subspace $\{0\}$ of $\R\sp k$ if and only if $p$ is an isolated point of $M$. In
fact, if $p$ is a limit point, then, given any $\{p_n\}$ in $M
\backslash \{p\}$ such that $p_n \to p$, the sequence $\big\{\alpha_n(p_n -
p)\big\}$, with $\alpha_n = 1/\|p_n - p\|$, admits a convergent subsequence
whose limit is a unit vector.

One can show that in the special and important case when $M$ is a
smooth manifold with (possibly empty)
boundary $\partial M$
(a \emph{$\partial$-manifold} for short),
this definition of tangent space is equivalent to the classical one
(see for instance \cite{Mi}, \cite{GuPo}).
Moreover, if $p \in \partial M$,
$C_pM$ is a closed half-space in $T_pM$ (delimited by
$T_p \partial M$), while
$C_pM = T_pM$ if $p \in M \backslash \partial M$. 

\smallskip

Let, as above, $M$ be a subset of $\R\sp k$.
We denote by $D$ a nontrivial closed real interval with $\max D = 0$;
that is, $D$ is either $(-\infty, 0]$ or $[-r, 0]$ with $r>0$.
By $C(D, M)$ we mean the metrizable space of the $M$-valued
continuous functions defined on $D$ with the
topology of the uniform convergence on compact subintervals of $D$.

Given a continuous function $x\cl J \to M$, defined on a real interval
$J$, and given $t \in \R$ such that $t+D \subseteq J$, we adopt the
standard notation $x_t \cl D \to M$ for the function defined by
$x_t (\theta) = x(t + \theta)$.

Let $h\cl \R \times C(D,M) \to \R \sp{k}$ be a
continuous map.
We say that
$h$ is a \emph{functional field on $M$} if 
$h(t,\varphi) \in T_{\varphi(0)}M$ for all $(t,\varphi) \in \R \times C(D,M)$.
In particular, $h$ will be said \emph{inward} (to $M$) if $h(t,\varphi) \in C_{\varphi(0)}M$ for all $(t,\varphi)$.
If $M$ is a closed subset of a boundaryless smooth
manifold $N \subseteq \R\sp k$, we will say that $h$ is \emph{away from $N \backslash M$} if $h(t,\varphi) \not \in C_{\varphi(0)} (N \backslash M)$ for all $(t,\varphi) \in \R \times C(D,M)$.
Notice that this condition is satisfied whenever the point $\varphi(0) \in M$ is not in the topological boundary of $M$ relative to $N$
since, in that case, $C_{\varphi(0)} (N \backslash M) = \emptyset$.

Let us consider a retarded functional differential equation (\textit{RFDE} for short) of the type
\begin{equation}
\label{equ-h}
x'(t) = h(t,x_t),
\end{equation}
where $h\cl \R \times C(D,M) \to \R\sp{k}$ is a functional field on $M$. 

By a \emph{solution} of \eqref{equ-h} we mean a continuous function
$x\cl J \to M$, defined on a real interval $J$ with $\inf J = -\infty$,
which verifies eventually the equality $x'(t) = h(t,x_t)$.
That is, $x$ is a solution of \eqref{equ-h} if there exists
$\bar t$, with $-\infty \leq \bar t < \sup J$,
such that $x$ is $C\sp{1}$ on the subinterval
$(\bar t,\sup J)$ of $J$ and verifies $x'(t) = h(t,x_t)$ for all
$t \in (\bar t,\sup J)$. 

Observe that, when $D = [ -r, 0]$,
there is a one-to-one correspondence between our notion of solution and the classical one which can be found e.g.\ in \cite{HL}
(see also \cite{Ol1}).
The correspondence is the one that assigns to any solution of 
\eqref{equ-h} its restriction to the interval
$[\bar t -r, \sup J)$.

\begin{remark} \label{rem-equiv}
We stress that it is possible to associate to any equation of the form \eqref{equ-h} with $D=[-r,0]$ an equivalent equation of the same type with $D = (-\infty, 0]$.
In other words, given a functional field $h\cl \R \times C([-r,0],M) \to \R\sp{k}$, there exists a functional field $g\cl \R \times C((-\infty, 0],M) \to \R\sp{k}$ such that the equation
\[
x'(t) = g(t,x_t)
\]
has the same set of solutions as \eqref{equ-h}.
To see this, it is enough to define 
$ g\cl \R \times C((-\infty, 0],M) \to \R\sp{k}$
by
\[
 g(t, \varphi) = h(t, \varphi|_{[-r,0]}).
\]
\end{remark}

Because of Remark \ref{rem-equiv}, from now on we will assume $D = (-\infty,0]$. That is, we will focus on RFDE's of the type
\begin{equation}
\label{equ-g-bis}
x'(t) = g(t,x_t),
\end{equation}
where $g\cl \R \times C((-\infty, 0],M) \to \R\sp{k}$ is a functional field on $M$.

\subsection{Initial value problem}

We are now interested in the following initial value problem:
\begin{equation}
\label{pb-g}
\left\{
\begin{array}{ll}
x'(t) = g(t,x_t), & t>0, \\
x(t) = \eta(t), & t \leq 0,
\end{array}
\right.
\end{equation}
where $M$ is a subset of $\R\sp k$,
$g\cl \R \times C((-\infty, 0],M) \to \R\sp{k}$ is a
functional field on $M$, and
$\eta\cl (-\infty, 0] \to M$ is a continuous map.

A \emph{solution} of problem \eqref{pb-g} is a solution $x\cl J \to M$
of \eqref{equ-g-bis} such that $\sup J >0$,
$x'(t) = g(t,x_t)$ for $t>0$,
and $x(t) = \eta (t)$ for $t \leq 0$.

The following technical lemma regards the existence of a persistent solution of
problem \eqref{pb-g}.

\begin{lemma}[\cite{BCFP5}] \label{lemma away from}
Let $M$ be a compact subset of a boundaryless smooth manifold
$N \subseteq \R\sp k$, and $g$ a functional field on $M$ which is away from $N
\backslash M$.
Suppose that $g$ is bounded. Then problem \eqref{pb-g} admits
at least one solution defined on the
whole real line.
\end{lemma}

{}From now on $M$ will be a compact $\partial$-manifold in $\R\sp k$. In this case
one may regard $M$ as a subset of a smooth boundaryless manifold $N$ of the
same dimension as $M$ (see e.g.\ \cite{Hi}, \cite{Mu}). It is not hard to
show that a functional field $g$ on $M$ is away from the complement $N \backslash
M$ of $M$ if and only if it is \emph{strictly inward}; meaning that $g$ is inward and
$g(t, \varphi) \not \in T_{\varphi(0)} \partial M$ for all $(t,\varphi) \in \R \times C((-\infty, 0], M)$ such that $\varphi (0) \in \partial M$.

\smallskip

We will consider the following assumption:
\begin{itemize}
\item[(H)] \label{acca} Given any compact subset $Q$ of $\R \times C((-\infty, 0], M)$, there exists $L \geq 0$ such that 
\[
\| g (t, \varphi) - g (t, \psi) \| \leq L \sup_{s\leq 0} \| \varphi (s) - \psi (s) \|
\]
for all $(t, \varphi), (t, \psi) \in Q$.
\end{itemize}

The following proposition regards existence and uniqueness of solutions of problem \eqref{pb-g} in the case when $g$
is inward, bounded, and verifies (H).
For the proof we refer to \cite{BCFP5}, in which the proposition is proved under a weaker assumption that (H).

\begin{proposition} \label{prop-esiuni}
Let $M \subseteq \R\sp k$ be a compact $\partial$-manifold and
$g$ an inward functional field on $M$.
Suppose that $g$ is bounded.
Then, problem \eqref{pb-g} admits a solution
defined on the whole real line.
Moreover, if $g$ verifies (H),
then the solution is unique. 
\end{proposition}

\subsection{Fixed point index}
Here we summarize the main properties of the fixed
point index in the context of 
absolute neighborhood retracts
(ANRs). 
Let $X$ be a metric
ANR and consider a locally compact (continuous) $X$-valued map $k$
defined on a subset $\mathcal D(k)$ of $X$. Given an open subset $U$
of $X$ contained in $\mathcal D(k)$, if the set of fixed points of
$k$ in $U$ is compact, the pair $(k,U)$ is called
\emph{admissible}. It is known that to any admissible pair $(k,U)$ we
can associate an integer $\ind_X(k,U)$ \mbox{--\;the} \emph{fixed point index}
of $k$ in \mbox{$U$\,--} which satisfies properties analogous to those of the
classical Leray--Schauder degree \cite{LS}. The reader can see for
instance \cite{Br}, \cite{Gra}, \cite{Nu1} or \cite{Nu3} for a
comprehensive presentation of the index theory for ANRs. As regards
the connection with the homology theory we refer to standard algebraic topology textbooks (e.g.\ \cite{Do}, \cite{Sp}).

We summarize for the reader's convenience the main properties of the index.
\begin{itemize}
\item[i)] ({\it Existence}) If $\ind_X(k,U)\neq 0$, then $k$ admits at
 least one fixed point in $U$.
\item[ii)] ({\it Normalization}) If $X$ is compact, then
 $\ind_X(k,X) = \Lambda (k)$, where $\Lambda (k)$ denotes the Lefschetz
 number of $k$.
\item[iii)] ({\it Additivity}) Given two disjoint open subsets $U_1$,
 $U_2$ of $U$ such that any fixed point of $k $ in $U$ is contained
 in $U_1\cup U_2$, then $\ind_X(k,U) = \ind_X(k,U_1)+ \ind_X(k,U_2)$.
\item[iv)] ({\it Excision}) Given an open subset $U_1$ of $U$ such
 that $k$ has no fixed points in $ U \backslash U_1$, then
 $\ind_X(k,U) = \ind_X(k,U_1)$.
\item[v)] ({\it Commutativity}) Let $X$ and $Y$ be metric
 ANRs. Suppose that $U$ and $V$ are open subsets of $X$ and $Y$
 respectively and that $k\cl U \to Y$ and $h\cl V \to X$ are locally compact
 maps. Assume that one of the sets of fixed points of $hk$ in $k\sp {-1}(V)$
 or $kh$ in $h\sp {-1}(U)$ is compact. Then the other set is compact as well
and $\ind_X(hk, k\sp {-1}(V)) = \ind_Y(kh, h\sp {-1}(U))$.
\item[vi)] ({\it Generalized homotopy invariance}) Let $I$ be a
 compact real interval and $\Omega$ an open subset of $X \times
 I$. For any $\lambda \in I$, denote $\Omega_\lambda = \{x \in X:
 (x,\lambda) \in \Omega\}$. Let $H\cl \Omega \to X$ be a locally compact
 map such that the set $\{(x,\lambda) \in \Omega: H(x,\lambda) = x\}$ is
 compact.
 Then $\ind_X(H(\cdot,\lambda),\Omega_\lambda)$ is independent of $\lambda$.
\end{itemize}

%%%%%
%%%%% section 3 - main results
%%%%% 

\section{Continuation results for first order equations}
\label{sect-continuation-1ord}
\setcounter{equation}{0}

Let $M$ be a compact $\partial$-manifold in $\R\sp k$.
Consider the RFDE
\begin{equation}
\label{equ-RFDE}
x'(t) = f(\lambda, t,x_t),
\end{equation}
where $f\cl [0,+\infty) \times \R \times C((-\infty,0],M) \to \R\sp{k}$ is an inward functional field on $M$
depending on a parameter $\lambda \in [0,+\infty)$, which is $T$-periodic in the second variable ($T$ being a positive real number). 

Suppose that $f$ is bounded on any set 
$[0, \bar \lambda] \times \R \times C((-\infty,0],M)$, $\bar \lambda>0$,
and that 
\begin{equation}
\label{equ-lambdazero}
f(0,t,\varphi) = 0 \quad \mbox{ for any }
(t, \varphi) \in \R \times C((-\infty,0],M).
\end{equation}
Moreover, assume that $f$ verifies the following assumption
(compare with assumption (H) on page \pageref{acca}):
\begin{itemize}
\item[(\~ H)] Given any $\bar \lambda >0$ and any compact subset $Q$ of $\R \times C((-\infty, 0], M)$, there exists $L \geq 0$ such that 
\[
\| f (\lambda, t, \varphi) - f (\lambda, t, \psi) \| \leq L \sup_{s\leq 0} \| \varphi (s) - \psi (s) \|
\]
for all $\lambda \in [0,\bar \lambda]$ and all $(t, \varphi), (t, \psi) \in Q$.
\end{itemize}

We will now prove a global continuation result (Theorem \ref{teo-ipotesilemma} below) for the equation \eqref{equ-RFDE} in the case when the Euler--Poincar\'e characteristic of $M$ is nonzero.
The results of this section, in particular Theorem \ref{teo-ipotesilemma}, will play a crucial role in the proof of Theorem \ref{teo-rami2ord} below.

In the sequel we will adopt the following notation.
By $C([-T,0],M)$ we shall mean the (complete) metric space of the continuous functions $\psi\cl [-T,0] \to M$ endowed with the metric induced by
the Banach space $C([-T,0],\R\sp{k})$.
We shall denote by $C_T(M)$ the set of the continuous $T$-periodic maps from $\R$ to $M$ with the metric induced by the Banach space $C_T(\R\sp{k} )$ of the continuous $T$-periodic $\R\sp{k} $-valued maps (with the standard supremum norm).

We will say that $(\lambda,x) \in[0,+ \infty)\times C_T(M)$
is a \emph{$T$-periodic pair} of \eqref{equ-RFDE} if
$x\cl \R \to M$ is a $T$-periodic solution of \eqref{equ-RFDE} corresponding
to $\lambda$. A $T$-periodic pair of the type $(0, x)$ is said to be
\emph{trivial}. In this case, because of assumption \eqref{equ-lambdazero},
the function $x$ is constant. 

A pair $(\lambda,\psi) \in [0,+ \infty)\times C([-T, 0],M)$ will be
called a \emph{source pair} (of \eqref{equ-RFDE}) if there exists
$x \in C_T(M)$ such that $x(t) = \psi (t)$ for all $t \in [-T, 0]$
and $(\lambda,x)$ is a $T$-periodic pair.
A source pair of the type $(0, \psi)$ will be called \emph{trivial}.
In this case the map $\psi$ is a constant $M$-valued map, since it is the restriction of a constant map defined on $\R$.

Observe that the map $\Theta\cl(\lambda,x)\mapsto (\lambda,\psi)$ which
associates to a $T$\mbox{-}periodic pair $(\lambda,x)$ the
corresponding source pair $(\lambda,\psi)$ is continuous,
$\psi$ being the restriction of $x$ to the interval $[-T,0]$.
Actually, $\Theta$ is an isometry 
between the set 
$\Sigma \subseteq [0,+ \infty) \times C_T(M)$
of the $T$-periodic pairs and the set 
$S \subseteq [0,+ \infty) \times C([-T, 0],M)$
of the source pairs.

It is not difficult to see that $\Sigma$ is a closed subset of the complete metric space $[0,+\infty) \times C_T(M)$.
Thus, because of Ascoli's Theorem, $\Sigma$ is a locally compact space, and this fact will turn out to be useful in the sequel.

Our first step is to define a Poincar\'e-type $T$-translation operator $P_{\lambda}$ on $C([-T, 0],M)$ whose fixed points are the restrictions to the interval $[-T,0]$ of the $T$-periodic solutions of \eqref{equ-RFDE}.
For this purpose, given $\psi \in C([-T, 0],M)$, we construct a suitable backward continuous extension $\widehat \psi$ of $\psi$ (see Proposition \ref{proposition-extension} below).
In this way, given $\lambda \geq 0$ and $t \in \R$, $f(\lambda, t, \widehat \psi)$ turns out to be a well defined vector in the tangent space $T_{\psi(0)}M$.
We point out that our operator $P_{\lambda}$ is different from the one in our paper \cite{BCFP5}, in which the assertion of Proposition 3.1 about $P_{\lambda}$ is imprecise.
In fact, the operator $P_{\lambda}$ considered in \cite[Proposition 3.1]{BCFP5} could admit fixed points which do not correspond to $T$-periodic solutions, contrary to what stated.
That assertion, however, holds true for the operator $P_{\lambda}$ defined below (see Proposition \ref{prop-fixed}).
Consequently, all the results in \cite{BCFP5} are correct, Proposition~3.1 apart.

In what follows, by a $T$-periodic map defined on $(-\infty, 0]$ we mean the restriction of a $T$-periodic map on $\R$.

\begin{proposition} \label{proposition-extension}
There exists a continuous map from $C([-T, 0],M)$ to $C((-\infty, 0],M)$, 
$\psi\mapsto\widehat\psi$, with the following properties:
\begin{itemize}
\item[{1.}] The restriction of $\widehat\psi$ to $[-T,0]$ coincides with $\psi$.
\item[{2.}] If $\psi(-T)=\psi(0)$, then $\widehat\psi$ is $T$-periodic.
\end{itemize}
\end{proposition}
\begin{proof}
Given $\psi \cl [-T,0] \to M$ continuous, consider the $T$-periodic backward extension $\psi\sp{\sim} \cl (-\infty,0] \to M$ of the restriction of $\psi$ to $(-T,0]$. This function is not continuous, unless $\psi(-T) = \psi(0)$. However, there exists (and is unique) a piecewise constant function $\psi\sp{-} \cl (-\infty,0] \to \R\sp{k}$, which is zero on $(-T,0]$, and which makes $\psi\sp{*} := \psi\sp{-} + \psi\sp{\sim}$ continuous. Notice that this new function depends continuously on $\psi$, with the topology on $C((-\infty, 0],\R\sp{k})$ of the uniform convergence on compact subintervals of $(-\infty,0]$. However, $\psi\sp{*}$ is not necessarily $M$-valued, unless $\psi(-T) = \psi(0)$. Therefore, we are going to define $\widehat \psi \in C((-\infty, 0],M)$ by ``pushing $\psi\sp{*}$ down on $M$''.

Since $M$ is a closed ANR in $\R\sp{k}$, there exists a retraction $\rho \cl U \to M$ defined on an open neighborhood $U \subseteq \R\sp{k}$ of $M$.
As $M \not= \R\sp{k}$, we may assume $\R\sp k \setminus U \not= \emptyset$.
Let $\sigma \cl \R\sp k \to [0,1]$ be the Urysohn function
\[
\sigma(x) = \frac{\dist(x, \R\sp k \setminus U)}{\dist(x, \R\sp k \setminus U)+\dist(x, M)}.
\]
On the metric space $C([-T, 0],M) \times (-\infty,0]$ consider the real function
\[
v(\psi,t) = \min \big\{\sigma(\psi\sp{*}(s)): t \leq s \leq 0 \big\}.
\]
This is clearly continuous.
We claim that the curve $t \mapsto \psi\sp{*}(v(\psi,t)t)$ is entirely contained in the open set $U$. 
In fact, given $t \in (-\infty,0]$, assume $\psi\sp{*}(v(\psi,t)t) \not\in U$.
Then $\sigma(\psi\sp{*}(s)) = 0$ for $s = v(\psi,t)t \in [t,0]$.
This implies $v(\psi,t) = 0$.
Hence $\psi\sp{*}(v(\psi,t)t) = \psi\sp{*}(0) = \psi(0) \in M$, a contradiction.
Thus, it makes sense to define $\widehat \psi \in C((-\infty, 0],M)$ by $\widehat \psi(t) = \rho(\psi\sp{*}(v(\psi,t)t))$. One can check that the map $\psi\mapsto\widehat\psi$ verifies properties $1$ and $2$.
\end{proof}

We are ready to define our operator $P_{\lambda}$ acting on the space $C([-T,0],M)$. For simplicity, from now on, this space will be denoted by $\widetilde M$.
Since $M$ is an ANR, it is not difficult to show (see e.g.\ \cite{EF}) that $\widetilde M$ is an ANR as well.

Let $\lambda \in [0,+ \infty)$ and define
\[
P_\lambda\cl \widetilde M \to \widetilde M
\]
by $P_\lambda (\psi)(s) = x(\lambda, \widehat \psi, s+T)$, where, given $\eta \in C((-\infty, 0],M)$, $t \mapsto x(\lambda,\eta,t)$ is the unique solution, ensured by Proposition \ref{prop-esiuni}, of the initial value problem
\begin{equation}
\label{pb-lf}
\left\{
\begin{array}{ll}
x'(t) = f(\lambda, t, x_t), & t > 0, \\
x(t) = \eta(t), & t \leq 0.
\end{array}
\right.
\end{equation}
The following two propositions regard some crucial properties of $P_\lambda$.

\begin{proposition} \label{prop-fixed}
The fixed points of $P_\lambda$ correspond to the $T$-periodic solutions of equation \eqref{equ-RFDE} in the following sense: $\psi$ is a fixed point of $P_\lambda$ if and only if it is the restriction to $[-T,0]$ of a $T$-periodic solution.
\end{proposition}
\begin{proof}
(if) Assume that $t \mapsto y(t)$ is a $T$-periodic solution of \eqref{equ-RFDE} and denote by $\psi$ its restriction to the interval $[-T,0]$.
Proposition \ref{proposition-extension} yields $y(t) = \widehat{\psi}(t)$ for $t \leq 0$. Thus, because of the uniqueness of the solution of the initial problem \eqref{pb-lf}, one gets $y(t) = x(\lambda,\widehat{\psi},t)$ or, equivalently, $y(t+T) = x(\lambda,\widehat{\psi},t+T)$, for all $t \in \R$.
In particular, for $-T \leq t \leq 0$, we obtain $\psi(t) = y(t) = y(t+T) = x(\lambda,\widehat{\psi},t+T) = P_\lambda(\psi)(t)$. That is, $\psi = P_\lambda(\psi)$.

(only if) Let $\psi$ be a fixed point of $P_\lambda$, and let, for short, $y(t)$ denote the solution $x(\lambda, \widehat \psi, t)$ of \eqref{equ-RFDE}.
Because of Proposition \ref{proposition-extension} (property $1$), the restriction of $y$ to the interval $[-T,0]$ coincides with $\psi$.
It remains to show that $y$ is a $T$-periodic function.

Consider first the restriction of $y$ to the half line $(-\infty,0]$.
Since this restriction coincides with $\widehat\psi$, according to Proposition \ref{proposition-extension} (property $2$), $y$ is $T$-periodic on $(-\infty,0]$ provided that $\psi(-T) = \psi(0)$.
As pointed out above, $\psi(t) = y(t)$ for $t \in [-T,0]$.
Moreover, since $\psi$ is a fixed point of $P_{\lambda}$, we get $\psi(t) = y(t+T)$ for $-T \leq t \leq 0$.
Thus, $\psi(-T) = \psi(0)$.

To prove that $y$ is $T$-periodic also on the interval $[0,+\infty)$, consider the function $z(t) \cl= y(t+T)$ and observe that it satisfies the same initial value problem as $y(t)$. Thus, because of the uniqueness of the solution of this problem, we finally get $y(t) = y(t+T)$ also for $t > 0$.
\end{proof}

\begin{proposition} \label{prop-compact}
The map $P\cl [0,+\infty) \times \widetilde M \to \widetilde M$, defined by
$(\lambda,\psi)\mapsto P_\lambda(\psi)$, is completely continuous.
\end{proposition}

\begin{proof}
Denote by $C(\R,M)$ the set of continuous maps from $\R$ into $M$.
This is a closed subset of the Fr\'echet space $C(\R,\R\sp{k})$ and, therefore, inherits the topology of uniform convergence on compact intervals, which is metrizable with complete metric.
Consider the map $(\lambda,\eta) \mapsto x(\lambda, \eta, \cdot)$ that to any $(\lambda,\eta) \in [0,+\infty) \times C((-\infty,0],M)$ assigns the solution in $C(\R,M)$ of the initial value problem \eqref{pb-lf}. Observe that, because of the Dominated Convergence Theorem, this map has closed graph.
Therefore, it is continuous if (and only if) it sends compact sets into relatively compact sets. Recall that, by assumption, the functional field $f$ is bounded on any set $[0, \bar \lambda] \times \R \times C((-\infty,0],M)$, $\bar \lambda>0$.
Consequently, because of Ascoli's Theorem, the map $(\lambda, \eta) \mapsto x(\lambda, \eta, \cdot)$ sends any such a set into a relatively compact subset of $C(\R,M)$. Thus, this map is actually completely continuous.
The assertion now follows from the fact that the operator $P$ is the composition of continuous maps, one of them (the above one) completely continuous.
\end{proof}

The following topological lemma is needed.

\begin{lemma}[\cite{FuPe93a}] \label{vecchiolemma}
Let $Z$ be a compact subset of a locally compact metric space $Y$. Assume
that any compact subset of $Y$ containing $Z$ has nonempty boundary. Then
$Y \backslash Z$ contains a connected set whose closure is not compact
and intersects $Z$.
\end{lemma}

\begin{remark} \label{nuovolemma}
Let $(Y,Z)$ be a topological pair satisfying the assumptions of Lemma \ref{vecchiolemma}. It is not hard to see that, given any compact subset $K$ of $Y$ containing $Z$, the pair $(Y,K)$ verifies the assumptions of Lemma \ref{vecchiolemma} as well.
\end{remark}

In the sequel, given a metric space $Y$, a subset $G$ of $\R \times Y$ and $\lambda \in \R$, we will
denote by $G_\lambda$ the slice $\{y \in Y : (\lambda,y) \in G\}$.

\begin{theorem}
\label{teo-ipotesilemma}
Let $M$ be a compact $\partial$-manifold with nonzero
Euler--Poincar\'e characteristic, 
and let $f\cl [0,+\infty) \times \R \times C((-\infty,0],M) \to \R\sp{k}$ be an inward parametrized functional field on $M$ which is $T$\mbox{-}periodic in the 
second variable.
Suppose that $f$ is bounded on any set $[0,\bar \lambda] \times \R \times C((-\infty,0],M)$ and verifies condition \eqref{equ-lambdazero}
 and assumption (\~ H).
Let
\[
\Sigma = \big\{ (\lambda,x) \in [0,+ \infty) \times C_T(M) :
(\lambda,x) \mbox{ is a $T$-periodic pair of \eqref{equ-RFDE}} \big\}.
\]
Then the pair $(\Sigma,\{0\} \times \Sigma_0)$ verifies the assumptions of Lemma \ref{vecchiolemma}.

Consequently, the equation \eqref{equ-RFDE} admits an unbounded connected set of nontrivial $T$-periodic pairs whose closure meets the set of the trivial $T$-periodic pairs.
In particular, for any fixed $\lambda$, the equation \eqref{equ-RFDE} has a $T$-periodic solution.
\end{theorem}

\begin{proof}
Let $S \subseteq [0,+ \infty) \times \widetilde M$ be the set of the source pairs of equation \eqref{equ-RFDE}.
Because of Proposition \ref{prop-fixed},
\[
S = \big\{ (\lambda,\psi) \in [0,+ \infty) \times \widetilde M :
P_\lambda(\psi) = \psi \big\}.
\]
Therefore, due to Proposition \ref{prop-compact}, $S$ is a locally compact metric space.
Let us show that the subset $\{0\} \times S_0$ of $S$ is compact.
In fact, because of \eqref{equ-lambdazero}, given $\psi \in \widetilde M$, we get $x(0,\widehat \psi,t) = \psi(0)$ for all $t \geq 0$.
Hence $P_0$ sends $\widetilde M$ onto the set of the constant $M$-valued functions, and this set, which coincides with $S_0$, can be identified with the compact manifold $M$.

We claim that the pair $(S,\{0\} \times S_0)$ verifies the assumptions of Lemma \ref{vecchiolemma}.
Assume, by contradiction, that there exists a compact set $\widehat S \subseteq S$ containing $\{0\}\times S_0$ and with empty boundary in $S$.
Thus, $\widehat S$ is also an open subset of the metric space $S$.
Hence, there exists a bounded open subset $\Omega$ of $[0,+ \infty) \times \widetilde M$ such that $\widehat S = \Omega \cap S$. 
Since $\widehat S$ is compact, the generalized homotopy invariance property of the fixed point index implies that $\ind_{\widetilde M} (P_\lambda,\Omega_\lambda)$ does not depend on $\lambda \in [0,+ \infty)$.
Moreover, the slice $\widehat S_\lambda = \Omega_\lambda \cap S_\lambda$ is empty for some $\lambda$.
This implies that $\ind_{\widetilde M} (P_\lambda, \Omega_\lambda) = 0$ for any $\lambda \in [0,+ \infty)$
and, in particular, $\ind_{\widetilde M} (P_0, \Omega_0) = 0$.
Since $\Omega_0$ is an open subset of $\widetilde M$ containing
$S_0$, by the excision property of the fixed point index we obtain
\begin{equation}
\label{excision}
 \ind_{\widetilde M} (P_0,\widetilde M) = \ind_{\widetilde M} (P_0,\Omega_0) = 0.
\end{equation}
As pointed out above, $P_0$ sends $\widetilde M$ onto the subset $S_0$ of the constant $M$-valued functions, which will be identified with $M$. According to this identification, the restriction $P_0|_M$ coincides with the identity $I_M$ of $M$.
Therefore, by the commutativity and normalization properties of the 
fixed point index, we get
\[
\ind_{\widetilde M} (P_0,\widetilde M) = \ind_M (P_0|_M, M)= \Lambda(I_M) = \chi(M) \neq 0,
\]
contradicting \eqref{excision}. Therefore, as claimed, $(S,\{0\} \times S_0)$ satisfies the assumptions of Lemma \ref{vecchiolemma}.

Now, to prove the assertion about the pair $(\Sigma,\{0\} \times \Sigma_0)$
observe that the map $\Theta\cl \Sigma \to S$ which associates to any $T$-periodic pair $(\lambda,x)$ the corresponding source pair $(\lambda,\psi)$ is a homeomorphism.
Moreover, the sets $\{0\} \times \Sigma_0$ and $\{0\} \times S_0$ correspond under this homeomorphism.
Hence, the pair $(\Sigma,\{0\} \times \Sigma_0)$ verifies the assumptions of
Lemma \ref{vecchiolemma} as well.

It remains to prove the last assertion.
Because of Lemma \ref{vecchiolemma}, there exists a connected subset
$A$ of $\Sigma$ whose closure in $\Sigma$ intersects $\{0\}\times M$
and is not compact.
Since $\Sigma$ is a closed subset of $[0,+ \infty) \times C_T(M)$, the
closure $\overline A$ of $A$ in $\Sigma$ is the same as in
$[0,+ \infty) \times C_T(M)$. Thus, $\overline A$ cannot be bounded since,
otherwise, it would be compact because of Ascoli's Theorem. Moreover,
since $C_T(M)$ is bounded, the set $A$ is necessarily unbounded in
$\lambda$. This implies, in particular, that the equation
\eqref{equ-RFDE} has a $T$-periodic solution for any $\lambda \ge 0$.
\end{proof}

%%%%%
%%%%% section 4 - ramo 2 ordine
%%%%% 

\section{Continuation results for motion equations}
\label{sect-continuation-2ord}
\setcounter{equation}{0}

Let $X \subseteq \R \sp{s}$ be a boundaryless manifold.
Given $q\in X$, let $(T_qX)\sp{\perp} \subseteq \R \sp{s}$ denote the normal space of $X$ at $q$. Since $\R\sp{s} = T_qX \oplus
(T_qX)\sp{\perp}$, any vector $u \in
\R\sp{s}$ can be uniquely decomposed into the sum of the \emph{parallel} (or
\emph{tangential}) \emph{component $u_\pi\in T_qX$ of $u$ at $q$} and the
\emph{normal component $u_\nu\in (T_qX)\sp{\perp}$ of $u$ at $q$}. By
\[
TX = \left\{ (q,v)\in \R\sp{s}\times \R\sp{s} : \; q\in X,\;v \in T_qX\right\}
\]
we denote the \emph{tangent bundle of $X$}, which is a smooth manifold
containing a natural copy of $X$ via the embedding $q \mapsto (q,0)$. The
natural projection of $TX$ onto $X$ is just the restriction (to $TX$ as
domain and to $X$ as codomain) of the projection of $\R\sp{s}\times \R\sp{s}$
onto the first factor.

It is known that, associated with $X \subseteq \R\sp{s}$, there exists a unique
smooth map
$R\cl TX \to
\R\sp{s}$, called the
\emph{reactive force} (or \emph{inertial reaction}), with the following
properties:
\begin{itemize}
\item [(a)] $R(q,v)\in (T_qX)\sp{\perp}$ for any $(q,v)\in TX$;
\item [(b)] $R$ is quadratic in the second variable;
\item [(c)] given $(q,v) \in TX$, $R(q,v)$ is the unique vector such that $(v, R(q,v))$ belongs to $T_{(q,v)}(TX)$;
\item [(d)] any $C\sp{2}$ curve $\gamma\cl (a,b) \to X$ verifies the condition
$\gamma''_{\nu}(t) = R(\gamma(t), \gamma'(t))$ for any $t\in (a,b)$, i.e.\ for each $t\in (a,b)$, the normal component $\gamma''_\nu(t)$ of $\gamma''(t)$ at $\gamma(t)$ equals $R(\gamma(t), \gamma'(t))$.
\end{itemize}

\smallskip

Let $F \cl \R \times C((-\infty, 0],X) \to \R\sp{s}$ be a functional
field which is $T$-periodic in the first variable.
Consider the following retarded functional motion equation on $X$,
depending on a parameter $\lambda \geq 0$:
\begin{equation}
\label{equ-motion-le}
x_\pi''(t) = 
\lambda \left( F (t,x_t) - \varepsilon x'(t) \right),
\end{equation}
where $x''_\pi (t)$ stands for the parallel component of the
acceleration $x''(t) \in \R\sp{s}$ at the point $x(t)$,
and $\varepsilon \geq 0$ is the frictional coefficient.
By properties (a) and (d) above, equation
\eqref{equ-motion-le} can be equivalently written as
\begin{equation}
\label{secorcomplete}
x''(t) = R(x(t), x'(t)) + \lambda \left( F (t,x_t) - \varepsilon x'(t) \right).
\end{equation}
Given $\lambda \geq 0$,
equation \eqref{secorcomplete} is equivalent to the RFDE
\begin{equation} \label{sys-motion}
\left\{
\begin{array}{l}
x'(t) = y(t),\\
y'(t) = R(x(t), y(t)) + \lambda \left( F (t,x_t) - \varepsilon y(t) \right)
\end{array}
\right.
\end{equation}
in the following sense: a function $x\cl J \to M$ is a solution of \eqref{secorcomplete} if and only if the pair $(x,x')$ is a solution of \eqref{sys-motion}.
Let us stress that system \eqref{sys-motion} is actually a RFDE, which can be written as
\[
(x'(t),y'(t))=G(\lambda, t, (x_t,y_t)),
\]
where the map $G\cl [0,+\infty) \times \R \times
C((-\infty, 0],TX) \to \R\sp{s}\times \R\sp{s}$ is defined by
\[
G (\lambda, t, (\varphi,\psi)) =
(\psi(0), R(\varphi(0),\psi(0)) + \lambda \left( F (t,\varphi) - \varepsilon \psi(0) \right)).
\]
We remark that $G$ is a parametrized functional field on $TX$.
Indeed, the condition
\[
G (\lambda, t, (\varphi,\psi)) \in T_{(\varphi(0),\psi(0))} TX
\]
is verified for all $(\lambda, t, (\varphi,\psi)) \in [0,+\infty) \times \R \times C((-\infty, 0],TX)$ (see, for example, \cite{Fu} for more details).

Observe that, if equation \eqref{secorcomplete} reduces to the so-called \emph{inertial equation}
\[
x''(t)=R(x(t),x'(t)),
\]
one obtains the geodesics of $X$ as solutions.

\smallskip

Given $\lambda \geq 0$, any $T$-periodic solution of
\eqref{equ-motion-le} is called a \emph{forced oscillation}
corresponding to the value $\lambda$ of the parameter.
In other words, a forced oscillation is a solution of the differential
equation \eqref{equ-motion-le} which belongs to $C\sp{1}_T(X)$,
the metric subspace of the Banach space $C\sp{1}_T(\R\sp{s})$ of the
$T$-periodic $C\sp{1}$ maps $x\cl \R \to X$.

We will say that $(\lambda,x) \in [0,+ \infty)\times C\sp{1}_T(X)$
is a \emph{$T$-forced pair} of \eqref{equ-motion-le} if
$x\cl \R \to X$ is a forced oscillation of \eqref{equ-motion-le}
corresponding to $\lambda$.
That is, $(\lambda,x) \in [0,+ \infty)\times C\sp{1}_T(X)$ 
is a $T$-forced pair of \eqref{equ-motion-le} if and only if 
$(\lambda,(x,x')) \in [0,+ \infty)\times C_T(TX)$
is a $T$-periodic pair of system \eqref{sys-motion}.

Denote by $\Gamma$ the closed subset of
$[0,+ \infty)\times C\sp{1}_T(X)$ of all the $T$-forced pairs of
\eqref{equ-motion-le}.
Notice that, by Ascoli's Theorem, the set $\Gamma$ is locally compact.
Given $q \in X$, we denote by $\bar q \in C\sp{1}_T(X)$ the constant
map $t \mapsto q$, $t \in \R$.
We will call \emph{trivial $T$-forced pair} a pair of the form
$(0,\bar q)$, $q \in X$, and we will say that $X$ is the
\emph{trivial solutions' manifold} of $\Gamma$.

An element $q_0 \in X$ will be called a \emph{bifurcation point} of the
equation \eqref{equ-motion-le} if every neighborhood of
$(0, \bar q_0)$ in $[0,+ \infty) \times C\sp{1}_T(X)$ contains a nontrivial
$T$-forced pair, i.e.\ an element of $\Gamma\setminus X$.
One can show that any nontrivial $T$-forced pair 
$(\lambda,x)$ of \eqref{equ-motion-le}, sufficiently close to 
$(0, \bar q_0)$, must have $\lambda >0$.
This is due to the fact that the solutions of the inertial
equation are geodesics of $X$, and there are no nontrivial closed geodesics too
close to a given point in a Riemannian manifold.

\smallskip

We are now in the position to state our main result.

\begin{theorem} \label{teo-rami2ord}
Let $X \subseteq \R\sp{s}$ be a compact boundaryless manifold
whose Euler--Poincar\'e characteristic $\chi (X)$ is different from
zero, and $F \cl \R \times C((-\infty, 0],X) \to \R\sp{s}$ a functional
field which is $T$-periodic in the first variable.
Suppose that $F$ is bounded and verifies (H), and let $\varepsilon\geq 0$ be given.
Then, the equation \eqref{equ-motion-le}
admits
an unbounded connected set of nontrivial $T$-forced pairs whose closure
meets the set of the trivial
$T$-forced pairs at some bifurcation point.
If, in addition, $\varepsilon$ is positive, then this connected set is necessarily unbounded with respect to $\lambda$.
\end{theorem}

\begin{proof}

\textit{Step 1.}
Assume first $\varepsilon>0$, and consider the following system:
\begin{equation}
\label{equ-sys-mi}
\left\{
\begin{array}{l}
x'(t) = \alpha(\mu) \, y(t),\\
y'(t) = \alpha(\mu) \left( R(x(t), y(t)) - \varepsilon (1-\alpha(\mu)) y (t) \right) +
\beta(\mu) \left( F (t,x_t) - \varepsilon y (t) \right),
\end{array}
\right.
\end{equation}
where
\[
\alpha(\mu) = \left\{
\begin{array}{ll}
\mu &\mbox{ if } 0 \leq \mu \leq 1,\\
1 &\mbox{ if } \mu \geq 1
\end{array}
\right.
\]
and
\[
\beta(\mu) = \left\{
\begin{array}{ll}
0 &\mbox{ if } 0 \leq \mu \leq 1,\\
\mu-1 &\mbox{ if } \mu \geq 1.
\end{array}
\right.
\]
In this way, when $0 \leq \mu \leq 1$ we get
\[
\left\{
\begin{array}{l}
x'(t) = \mu \, y(t),\\
y'(t) = \mu \left( R(x(t), y(t)) - \varepsilon (1-\mu) y (t) \right) 
\end{array}
\right.
\]
and when $\mu \geq 1$ we obtain
\[
\left\{
\begin{array}{l}
x'(t) = y(t),\\
y'(t) = R(x(t), y(t)) +
(\mu-1) \, \left( F (t,x_t) - \varepsilon y (t) \right).
\end{array}
\right.
\]

Define $G\cl [0,+\infty) \times \R \times C((-\infty, 0],TX) 
\to \R\sp{s}\times \R\sp{s}$
by 
\[ 
G (\mu, t,(\varphi,\psi)) =
\Big( \alpha(\mu) \psi(0), \,
\alpha(\mu) \big( R(\varphi(0),\psi(0)) - \varepsilon (1-\alpha(\mu)) \psi (0) \big) 
+ \beta(\mu) \big( F (t,\varphi) - \varepsilon \psi(0)\big) \Big).
\] 
Clearly, $G$ is a parametrized $T$-periodic functional field on $TX$, which verifies assumption (\~H) since $F$ verifies (H).

Given $c>0$, set
\[
M_c = \big\{ (q,v) \in TX : \|v\| \leq c \big\}.
\]
It is not difficult to show that $M_c \subseteq TX$ is a compact 
$\partial$-manifold
in $\R\sp{s} \times \R\sp{s}$ with boundary
\[
\partial M_c = \big\{ (q,v) \in M_c : \|v\| = c \big\}.
\]

Let $G_c$ be the restriction of the map $G$ to 
$[0,+\infty) \times \R \times C((-\infty, 0],M_c)$.
Clearly, $G_c$ is a parametrized $T$-periodic functional field on $M_c$ which verifies (\~ H) since so does $G$.
Moreover, given any $\mu_0>0$, the map $G_c$ is bounded on $[0,\mu_0]
\times \R \times C((-\infty, 0],M_c)$. 
To see this notice that $0\leq \alpha(\mu) \leq1$ and
$0\leq \beta(\mu) \leq \beta(\mu_0)$.
Now, the map $F$ is bounded by assumption,
hence there is $K$ such that $\|F(t,\varphi)\| \leq K$
for all $(t,\varphi)$.
Moreover, the compactness of $M_c$ implies that 
there is $K_1$ such that $\| R(q,v)\| \leq K_1$ for any
$(q,v) \in M_c$.
Therefore, given $(\mu,t,(x,y)) \in [0,\mu_0]
\times \R \times C((-\infty, 0],M_c)$,
one has 
$\| \alpha(\mu) \psi(0) \| \leq c$,
\[
\| \alpha(\mu) \big( R(\varphi(0),\psi(0)) - \varepsilon (1-\alpha(\mu)) \psi (0) \big)\| \leq K_1+ \varepsilon c
\]
and
\[
\| \beta(\mu) \big( F (t,\varphi) - \varepsilon \psi(0) \big) \| \leq 
\beta(\mu_0) (K + \varepsilon c).
\]
This shows that $G_c$ is bounded on $[0,\mu_0]
\times \R \times C((-\infty, 0],M_c)$,
being the sum of bounded maps.

We claim that, if $c>0$ is large enough, then $G_c$ is 
 inward on $M_c$. 
To see this, observe that the tangent cone of $M_c$
at $(q,v) \in \partial M_c$ is the half subspace of $T_{(q,v)}M_c$ given by
\[
C_{(q,v)}M_c =
\big\{(\dot q, \dot v) \in T_{(q,v)}(TX):
\langle v, \dot v \rangle \leq 0 \big\},
\]
where $\langle \cdot, \cdot \rangle$ denotes the inner product in
$\R\sp{s}$.
Thus we have to show that, if $c>0$ is large enough, then
$G_c (\mu, t,(\varphi,\psi))$ belongs to 
$C_{(\varphi(0),\psi(0))} M_c$
for any $\mu \in [0,+\infty)$, $t \in \R$, and any pair
$(\varphi,\psi) \in C((-\infty, 0],M_c)$
such that $(\varphi(0),\psi(0)) \in \partial M_c$.
In other words, we need to prove that for any $\mu$, $t$ and any pair $(\varphi,\psi)$ such that $\| \psi(0) \| = c$, with $c$ to be chosen, we have
\[
\Big\langle \psi(0), 
\alpha(\mu) \big( R(\varphi(0),\psi(0)) - \varepsilon (1-\alpha(\mu)) \psi (0) \big) 
+ \beta(\mu) \big( F (t,\varphi)-\varepsilon \psi(0) \big) \Big\rangle \leq 0.
\]
That is,
\[
\alpha(\mu) \Big(
\big\langle \psi(0),R(\varphi(0),\psi(0)) \big\rangle 
- \varepsilon (1-\alpha(\mu)) \big\langle \psi(0),\psi(0) \big\rangle 
\Big) +
\beta(\mu) \Big(
 \big\langle \psi(0), F (t,\varphi) \big\rangle -
\varepsilon \big\langle \psi(0),\psi(0) \big\rangle 
\Big) \leq 0.
\]
Observe first that $\big\langle \psi(0),R(\varphi(0),\psi(0)) \big\rangle = 0$ since
$R(\varphi(0),\psi(0))$ belongs to $(T_{\varphi(0)}X)\sp{\perp}$ and that, clearly, $\big\langle \psi(0),\psi(0) \big\rangle = c\sp2$.
Assume now $0\leq \mu \leq 1$. Then, $\alpha (\mu)=\mu$ and $\beta (\mu)=0$, and thus
\[
\Big\langle \psi(0), 
\alpha(\mu) \big( R(\varphi(0),\psi(0)) - \varepsilon (1-\alpha(\mu)) \psi (0) \big) 
\Big\rangle =
- \varepsilon \mu (1-\mu) c\sp2 \leq 0.
\]
On the other hand, when $\mu \geq 1$ one has $\alpha (\mu)=1$ and
$\beta (\mu)=\mu -1$. In addition, one has
\[
\big\langle \psi(0), F (t,\varphi) \big\rangle \leq
\|\psi(0)\| \, \| F (t,\varphi)\| \leq K \|\psi(0)\|,
\]
recalling that $K$ is such that
$\| F (t,\varphi) \| \leq K$
for all
$(t,\varphi) \in \R \times C((-\infty, 0],X)$.
Consequently, 
\[
\Big\langle \psi(0), 
\alpha(\mu) \, R(\varphi(0),\psi(0)) 
+ \beta(\mu) \big( F (t,\varphi)-\varepsilon \psi(0) \big) \Big\rangle =
(\mu -1) \Big\langle \psi(0), F (t,\varphi)-\varepsilon \psi(0) \Big\rangle \leq
(\mu -1) (Kc-\varepsilon c\sp2).
\]
This shows that, if we choose $c \geq K/\varepsilon$, then $G_c$ is an inward parametrized functional field on $M_c$, as claimed
(here the condition $\varepsilon >0$ is crucial).

Let now
\[
\Sigma = \big\{ (\mu,(x,y)) \in [0,+ \infty) \times C_T(M_c) :
(\mu,(x,y)) \mbox{ is a $T$-periodic pair of \eqref{equ-sys-mi}} \big\}.
\]
Observe that the slice $\Sigma_0$ coincides with $M_c$.
Let us show that $\Sigma_\mu = X \times \{0\}$
for $0 < \mu < 1$.
To see this recall that, when $0< \mu <1$,
system \eqref{equ-sys-mi} becomes
\[
\left\{
\begin{array}{l}
x'(t) = \mu \, y(t),\\
y'(t) = \mu \left( R(x(t), y(t)) - \varepsilon \, (1- \mu) y(t) \right).
\end{array}
\right.
\]
Let $\delta(t) = \|y(t)\|\sp2$, $t\in\R$.
One has
\[
\delta'(t) = 2\langle y(t),y'(t) \rangle = 2\mu\big\langle y(t),\left( R(x(t), y(t)) - \varepsilon \, (1- \mu) y(t) \right) \big\rangle
= - 2 \varepsilon \mu (1- \mu) \|y(t)\|\sp2
= - a \delta(t),
\]
where $a=2 \varepsilon \mu (1- \mu)>0$. This proves that if $(\mu, (x,y))$ is a $T$-periodic pair with $0< \mu <1$,
then $y(t) =0$ for any $t$.
That is, $(x,y)$ belongs to $X \times \{0\}$.

Consider the compact set
\[
\Upsilon = \left(\{0\} \times M_c \right) \cup
\left([0,1] \times X \times \{0\} \right) \subseteq \Sigma.
\]
We claim that the pair $(\Sigma,\Upsilon)$
verifies the assumptions of Lemma \ref{vecchiolemma}.
To this end, observe first that
$\chi(M_c) = \chi(X)$ since $M_c$ and $X$ are homotopically
equivalent ($X$ being a deformation retract of $TX$).
Consequently, since $\chi(X) \neq 0$ by assumption,
we have $\chi(M_c) \neq 0$.
Thus, given $c>K/\varepsilon$, we can apply Theorem \ref{teo-ipotesilemma}
with $M=M_c$ and $f=G_c$, and we get that the pair $(\Sigma,\{0\}\times \Sigma_0)$
verifies the assumptions of Lemma \ref{vecchiolemma}.
Hence, being $\Upsilon$ a compact subset of $\Sigma$ containing $\{0\}\times \Sigma_0 = \{0\} \times M_c$, by Remark \ref{nuovolemma} the pair $(\Sigma,\Upsilon)$
verifies the assumptions of Lemma \ref{vecchiolemma}, as claimed.
Hence, there exists a connected subset
$A$ in $\Sigma \setminus \Upsilon$ whose closure in $\Sigma$ intersects $\Upsilon$
and is not compact.

Recall that, for $\mu \geq 1$, system \eqref{equ-sys-mi} becomes
\[
\left\{
\begin{array}{l}
x'(t) = y(t),\\
y'(t) = R(x(t), y(t)) +
(\mu -1) \left( F (t,x_t) - \varepsilon y (t) \right).
\end{array}
\right.
\]

Clearly, up to the change of variable $\lambda = \mu -1$, this coincides with system \eqref{sys-motion}, which is equivalent to equation \eqref{equ-motion-le}.
The above argument shows that system \eqref{sys-motion}
admits an unbounded connected set
$A \subseteq [0,+\infty) \times C_T(M_c)$
of nontrivial $T$-periodic pairs
whose closure in $[0,+\infty) \times C_T(M_c)$, which is the same as in $[0,+\infty) \times C_T(TX)$,
intersects $\{0\} \times C_T(TX)$
in a nonempty subset of $\{0\} \times X \times \{0\}$.
Notice that $A$ is necessarily unbounded with respect to $\lambda$, being $C_T(M_c)$ a bounded metric space.

Finally, to prove the assertion in the case $\varepsilon >0$,
consider the map
$\Pi\cl [0,+\infty) \times C_T(TX) \to [0,+\infty) \times C\sp{1}_T(X)$
associating to a $T$-periodic pair $(\lambda,(x,y))$ of system
\eqref{sys-motion} the $T$-forced pair $(\lambda,x)$ of equation
\eqref{equ-motion-le}.
Observe that $\Pi$ is continuous
with inverse given by $(\lambda,x) \mapsto (\lambda,(x,x'))$.
Moreover, the restriction of $\Pi$
to $\{0\} \times X \times \{0\}$ as domain and to $\{0\} \times X$ as
codomain can be regarded as the identity of $X$. 
Hence $B = \Pi(A)$ is an unbounded connected set of nontrivial $T$-forced pairs whose closure
in $[0,+ \infty)\times C\sp{1}_T(X)$
meets the set of the trivial
$T$-forced pairs.
Moreover, $B$ is clearly unbounded with respect to $\lambda$ since so is $A$.

\smallskip
\textit{Step 2.} Assume now $\varepsilon =0$.
Apply Lemma \ref{vecchiolemma} to the pair $(\Gamma,\{0\}\times X)$, where $\Gamma$ denotes the set of the $T$-forced pairs of equation
\begin{equation}
\label{equ-motion-l0}
x_\pi''(t) = 
\lambda F (t,x_t).
\end{equation}

As already pointed out, $\Gamma$ is a closed, locally
compact subset of $[0,+ \infty) \times C\sp{1}_T(X)$.
Assume, by contradiction, that there exists a compact set $\widehat
\Gamma \subseteq \Gamma$ containing $\{0\}\times X$ and with empty boundary in
the metric space $\Gamma$. Thus, $\widehat \Gamma$ is also an
open subset of $\Gamma$ and, consequently, $\Gamma
\backslash
\widehat \Gamma$ is closed in $[0,+ \infty) \times C\sp{1}_T(X)$. Hence, there
exists a bounded open subset $W$ of $[0,+ \infty) \times C\sp{1}_T(X)$ such that
$\widehat \Gamma \subseteq W$ and $\partial W \cap \Gamma = \emptyset$.

Let now $\{\varepsilon_n\}$ be a sequence such that 
 $\varepsilon_n >0$ and $\varepsilon_n\to 0$.
Given any $n \in \N$,
let $\Gamma_n$ denote the set of the $T$-forced pairs of the 
equation
\[
x_\pi''(t) = 
\lambda \left( F (t,x_t) - \varepsilon_n x'(t) \right).
\]
Since $W$ is bounded and contains $\{0\}\times X$, 
as a consequence of Step 1, for any $n \in \N$ there exists a pair $(\lambda_n,x\sp{n}) \in \Gamma_n \cap
\partial W$. We may assume $\lambda_n \to \lambda_0$ and, by Ascoli's Theorem,
$x\sp{n} \to x\sp 0$ in $C\sp{1}_T(X)$.
Thus, $x\sp 0$ is a $T$-periodic solution of the equation
\[
x_\pi''(t) = 
\lambda_0 \, F (t,x_t).
\]
That is, $(\lambda_0,x\sp 0)$ is a $T$-forced pair of \eqref{equ-motion-l0}
and, consequently, $(\lambda_0,x\sp 0)$ belongs to $\partial W \cap \Gamma$,
which is a contradiction. Therefore, by Lemma \ref{vecchiolemma} one can find
a connected branch $B$ of nontrivial $T$-forced pairs of
\eqref{equ-motion-l0} whose closure in $\Gamma$ (which is the same
as in $[0,+ \infty) \times C \sp{1}_T(X)$) intersects
$\{0\}\times X$ and is not compact. 
Finally, $B$ cannot be bounded since,
otherwise, because of Ascoli's Theorem, its closure would be compact. 
This completes the proof.
\end{proof}

As a straightforward consequence of Theorem \ref{teo-rami2ord} we get the following existence result (see \cite{BCFP5}).

\begin{corollary}
\label{coroll-esist}
Let $X$ and $F$ be as in Theorem \ref{teo-rami2ord} and let $\varepsilon>0$.
Then, the equation
\begin{equation}
\label{equ-motion-epsilon}
x_\pi''(t) = F (t,x_t) - \varepsilon x'(t)
\end{equation}
has a $T$-periodic solution.
\end{corollary}

As we already pointed out in the Introduction, when $\varepsilon =0$, 
we do not know whether or not the equation
\begin{equation}
\label{equ-motion-0}
x_\pi''(t) = F (t,x_t)
\end{equation}
has a forced oscillation.
Actually, we proved that \eqref{equ-motion-le} admits an unbounded bifurcating branch, but we cannot assert that such a branch is unbounded with respect to $\lambda$.
As far as we know, the problem of the existence of forced oscillations of \eqref{equ-motion-0} is still open, even in the undelayed situation.
In the particular case of the spherical pendulum, i.e.\ $X = S \sp 2$, the existence of forced oscillations for equations without delay has been proved by the last two authors in \cite{FuPe91}, and this result has been extended in \cite{FuPe93o} to the case $X = S\sp{2n}$.

%%%%%
%%%%% BIBLIOGRAPHY
%%%%%

\bigskip
\noindent
\small{{\tt
Pierluigi Benevieri, Massimo Furi, and Maria Patrizia Pera,
\\
Dipartimento di Matematica Applicata ``Giovanni Sansone''
\\
Universit\`a degli Studi di Firenze
\\
Via S. Marta 3
\\
I-50139 Firenze, Italy

\medskip
\noindent
Alessandro Calamai
\\
Dipartimento di Scienze Matematiche\\
Universit\`a Politecnica delle Marche\\
Via Brecce Bianche
\\
I-60131 Ancona, Italy.

\medskip
\noindent
e-mail addresses:
\\
pierluigi.benevieri@unifi.it
\\
calamai@dipmat.univpm.it
\\
massimo.furi@unifi.it
\\
mpatrizia.pera@unifi.it
}}

\end{document}